\documentclass[a4paper]{article}
\usepackage{ayv_paperstyle}
\newcommand{\ptd}{\mathcal{P}^2(\td)}
\newcommand{\tdu}{\td\times U}
\newcommand{\ptdu}{\mathcal{P}^2(\td\times U)}
\newcommand{\tdv}{\td\times V}

\newcommand{\tdubc}[1][c]{\td\times U\times \mathbb{B}_{#1}}
\newcommand{\ptdubc}[1][c]{\mathcal{P}^2(\td\times U\times \mathbb{B}_{#1})}

\newcommand{\ptdvbc}[1][c]{\mathcal{P}^2(\td\times V\times \mathbb{B}_{#1})}
\newcommand{\pctu}[2]{\mathcal{P}^2(\mathcal{C}_{#1,#2}\times U)}

\newcommand{\ctu}[2]{\mathcal{C}_{#1,#2}\times U}
\newcommand{\hae}[1]{\hat{e}_{#1}}
\newcommand{\et}[1]{e^1_{#1}}
\newcommand{\ud}{\mathrm{u\text{-}d}}
\newcommand{\vd}{\mathrm{v\text{-}d}}
\title{A stability property in  mean field type differential games}
\author{Yurii Averboukh\footnote{Krasovskii Institute of Mathematics and Mechanics,\ \ 
\email{ayv@imm.uran.ru}} }
\date{}

\begin{document}
	\maketitle
	
\begin{abstract}
	The paper is concerned with the feedback approach to the deterministic mean field type differential games. Previously, it was shown that suboptimal strategies in the mean field type differential game can constructed based on functions of time and probability satisfying the stability condition. This property realizes the dynamic programming principle for the constant control of one player. We present the infinitesimal form of this condition involving analogs of the directional derivatives. In particular, we obtain the  characterization of the value function of the deterministic mean field type differential game in the terms of directional derivatives and the set of  directions feasible by virtue of the dynamics of the game.
	\msccode{93C30, 49L20, 46G05, 93A15}
	
\end{abstract}

\section{Introduction} The theory of mean field type differential game studies  control systems consisting of a large number of similar agents with the mean field interaction between them governed by two players with opposite purposes. This problem is a natural extension of the theory of mean field type control systems dealing with the case of only one decision maker. Let us emphasize that the state space for the mean field type differential games and the mean field type control systems is the space of probability measures. This space is only metric. 

The theory of mean field type control systems started with  paper~\cite{ahmed_ding_controlld}. Nowadays, the mean field field type control theory is developed for the case when  the dynamics of each agent is given by SDE (see~\cite{Bensoussan_Frehse_Yam_book} and reference therein). For such type of systems the necessary optimality conditions based on forward-backward SDEs were obtained in ~\cite{Andersson_Djehiche_2011},~\cite{Buckdahn_Boualem_Li_PMP_SDE},~\cite{Carmona_Delarue_PMP}. The existence result of the optimal control is also proved~\cite{Bahlali_Mezerdi_Mezerdi_existence}. The dynamic programming principle for mean field type control systems was discussed in~\cite{Bensoussan_Frehse_Yam_equation},~\cite{Lauriere_Pironneau_DPP_MF_control},~\cite{Pham_Wei_2015_DPP_2016},~\cite{Pham_Wei_2015_Bellman}.
Additionally, let us mention papers~\cite{Marigonda_Cavagnari},~\cite{Marigonda_et_al_2015} concerning with the case of deterministic mean field type control systems.

The mean field type differential games previously were studied in~\cite{Averbokh_mfdg},~\cite{Cosso_Pham},~\cite{Djehiche_Hamdine}. Recall that as in the case of finite dimensional differential games one can formalize mean field type differential game using either nonanticipative or feedback strategies. 

A nonanticipative strategy is a mapping assigning to a control of a player a control of the other player satisfying the feasibility condition. Notice that this approach assumes that the players observe the control of each other. The nonanticipative strategies were introduced for the finite dimensional differential game for the mean field type differential games in~\cite{elliot_kalton},~\cite{varaya_lin}. This approach was developed in~\cite{Cosso_Pham}. In that paper the existence result for the value function is proved and the dynamic programming is presented. 

The assumption that the players have information only about the current state leads to the feedback formalization. For the finite dimensional differential game feedback strategies were introduced by Krasovskii and Subbotin~\cite{NN_PDG_en}. The extension of their approach to the mean field type differential is presented in~\cite{Averbokh_mfdg}. The design of feedback strategies can be performed using so called $u$- and $v$-stable functions defined on the product of the time interval and the state space and taking  values in the set of real values. Given a $u$-stable (respectively, $v$-stable) function, the first (respectively, second) player can construct a suboptimal strategy guaranteeing the reward greater (respectively, smaller) than the value of the given function at the initial position~\cite{NN_PDG_en} (see also~\cite{Averbokh_mfdg} for the mean field type differential games). The $u$- (respectively $v$-) stability property means that the epigraph (respectively, hypograph) is viable with respect to the dynamics corresponding to the constant control of the second (respectively, first) player. Notice that the stability property realizes the dynamic programming principle for the frozen control of one player. It is proved that  the value function is simultaneously $u$- and $v$-stable~\cite{Averbokh_mfdg},~\cite{NN_PDG_en},~\cite{Subb_book}. 

The approaches based on nonaticipative and feedback strategies should be equivalent. This statement is proved for the case of finite dimensional differential games in~\cite{subb_chen}. Unfortunately, up to now this equivalence is not obtained for the mean field type differential games.

Recall that the dynamic programming reduces the original control problem to the Hamilton-Jacobi PDE,  which has no smooth solution in the general case. In particular, it was proved for the finite dimensional case that the $u$-stable (respectively, $v$-stable) function  is a supersolution (respectively, subsolution) of the corresponding Hamilton-Jacobi equation~\cite{Subb_book}.  For the finite dimensional case one can use two equivalent tools of nonsmooth analysis to define the viscosity solution of the Hamilton-Jacobi PDE~\cite{bardi},~\cite{frankowska},~\cite{Subb_book},~\cite{vinter_wolenski}. First is based on sub- and superdifferentials, whereas the second involves directional derivatives. 

Nowadays, only the notions of sub- and superdifferentials are introduced for the functions of probability. Lions in~\cite{lions_lecture} proposed the   extrinsic approach which is based on the lifting a probability measure to a random variable.  The intrinsic  definition of sub- and superdifferentials was introduced in~\cite{Ambrosio}. The link between these two approaches is discussed in~\cite{Gangbo_Tudorascu}. Solutions of Hamilton-Jacobi PDEs in the space of probabilities in the framework of the extrinsic approach were studied  in ~\cite{Carmona_master},~\cite{Cosso_Pham},~\cite{Gangbo_Nguen_Tudorascu},~\cite{Gangbo_Tudorascu},~\cite{Pham_Wei_2015_Bellman}. The intrinsic sub- and superdifferentials were also used for this class of equations (see~\cite{Marigonda_Cavagnari},~\cite{Marigonda_et_al_2015},~\cite{Marigonda_Quincampoix}). 

%The paper aims to provide infinitesimal forms of the properties of $u$- and $v$-stability. In particular we obtain the characterization of the value function of mean field type differential game. Previously, 

The paper aims to extend the approach involving  directional derivatives to the case of mean field type differential games. The main result of the paper  is the   infinitesimal forms of  $u$- and $v$-stability properties for the deterministic mean field type differential game expressed in the terms of directional derivatives. This statement is a modification of the famous viability theorem ~\cite{Aubin} (see also~\cite{Averboukh_viability} for the viability theorem for the mean field type control systems). In particular, we get the characterization of the value function in the terms of directional derivatives. 

The paper is organized as follows. In Section~\ref{sect:preliminaries} we introduce the general notation used in the paper. Furthermore,  in this section we present the feedback formalization of the deterministic mean field type differential game and recall the  link between the stability property and the value function of the game. The main result of the paper is presented in Section~\ref{sect:main_result}. Section~\ref{sect:equivalence} contains the characterization of flows of propositions produced by distributions of constant controls. This reformulates the stability condition in the terms of differential inclusions. The properties of the shift operator in the space of probabilities used in the proof of the sufficiency part of the viability theorem are given in Section~\ref{sect:shift_operator}. Finally, the sufficiency and necessity parts of the main result are proved in Sections~\ref{sect:proof:sufficiency},~\ref{sect:proof:necessity} respectively.

\section{Preliminaries}\label{sect:preliminaries}
First, let us set down the notation for the paper.
\begin{itemize}
	\item If $X_1,\ldots, X_n$ are sets, $i,j\in \{1,\ldots,n\}$, $x=(x_1,\ldots,x_n)\in X_1\times\ldots\times X_n$, then $$\mathrm{p}^i(x)\triangleq x_i,\ \ \mathrm{p}^{i,j}(x)\triangleq (x_i,x_j).$$
	\item If $(\Omega_1,\Sigma_1)$, $(\Omega_2,\Sigma_2)$ are measurable spaces, $m$ is a probability on $\Sigma_1$, $h:\Omega_1\rightarrow\Omega_2$ is measurable, then $h_\#m$ stands for the push-forward measure defined by the rule: for $\Upsilon\in\Sigma_2$,
	$$(h_\#m)(\Upsilon)\triangleq m(h^{-1}(\Upsilon)). $$
	\item Given two metric spaces $(X,\rho_X)$ and $(Y,\rho_Y)$, we assume that the metric on $X\times Y$ is $$\rho_{X\times Y}((x',y'),(x'',y''))\triangleq \left[\left(\rho_X(x',x'')\right)^2+\left(\rho_Y(y',y'')\right)^2\right]^{1/2}. $$
	\item If $(X,\rho_X)$ is a separable metric space, then $\mathcal{P}(X)$ stands the set of Borel probabilities on $X$. We endow $\mathcal{P}(X)$ with the topology of narrow convergence i.e. the sequence of probabilities $\{m_n\}_{n=1}^\infty$ converges narrowly to $m$ if, for any $\phi\in C_b(X)$,
	$$\int_{X}\phi(x)m_n(dx)\rightarrow \int_X\phi(x)m(dx)\text{ as }n\rightarrow\infty. $$
	\item $\mathcal{P}^2(X)$ denotes the set of probabilities $m\in\mathcal{P}(X)$ such that, for some (equivalently, any) $x_*\in X$,
	$$\int_X(\rho_X(x,x_*))^2 m(dx)<\infty. $$
	\item The 2-Wasserstein metric on $\mathcal{P}^2(X)$ is defined as follows: if $m_1,m_2\in\mathcal{P}_2(X)$, then
	\begin{equation}\label{notation:intro:wasserstein}
	W_2(m_1,m_2)=\left[\inf_{\pi\in \Pi(m_1,m_2)}\int_{X\times X}(\rho_X(x_1,x_2))^2\pi(d(x_1,x_2))\right]^{1/2}, 
	\end{equation} where $\Pi(m_1,m_2)$ stands for the set of transport plans between $m_1$ and $m_2$ i.e. probabilities on $X\times X$ with the marginals equal to $m_1$ and $m_2$. Notice that if $W_2(m_n,m)\rightarrow 0$ as $n\rightarrow \infty$, then the sequence $\{m_n\}$ converges to $m$ narrowly. When $X$ is compact, the converse is also true and $\mathcal{P}^2(X)$ is compact itself..
	\item Denote by $\Pi^0(m_1,m_2)$ the set of all plans $\pi\in \Pi(m_1,m_2)$ minimizing the right-hand side in (\ref{notation:intro:wasserstein}).
	\item If $(X,\rho_X)$, $(Y,\rho_Y)$ are separable metric spaces, $m$ is a measure on $X$, then denote by $\Lambda(X,m,Y)$ the set of measures on $(X\times Y)$ with the marginal on $X$ equal to $m$. By the disintegration theorem, given measure $\alpha\in \Lambda(X,m,Y)$, there exists, a weakly measurable family of probabilities $\alpha(\cdot|x)\in\mathcal{P}(Y)$ such that, for any $\phi\in C_b(X\times Y)$,
	\begin{equation}\label{prel:equality:disintegration}
	\int_{X\times Y}\phi(x,y)\alpha(d(x,y))=\int_X\int_Y\phi(x,y)\alpha(dy|x)m(dx).
	\end{equation} If $\alpha'(dy|x)$, $\alpha''(dy|x)$ both satisfy (\ref{prel:equality:disintegration}), then $\alpha'(\cdot|x)=\alpha''(\cdot|x)$ for $m$-a.e. $x\in X$. Conversely, given a weakly measurable family of probabilities $\alpha(dy|x)$, (\ref{prel:equality:disintegration}) defines the unique measure $\alpha\in \Lambda(X,m,Y)$. Thus, we will identify a measure $\alpha\in \Lambda(X,m,Y)$ with the class of equivalence containing families of probabilities $\alpha(dy|x)$ satisfying (\ref{prel:equality:disintegration}). 
	\item Let $m,m'\in\mathcal{P}(X)$, $\pi\in \Pi(m',m)$, $\alpha\in\Lambda(X,m,Y)$.  Define the composition of $\pi$ and $\alpha$ $\pi*\alpha\in \Lambda(X,m',Y)$ by the following rule: for $\phi\in C_b(X,Y)$, 
	\begin{equation}\label{notation:intro:composition}
	\int_{X\times Y}\phi(x',y)(\pi*\alpha)(d(x',y))\triangleq \int_{X\times X}\int_Y\phi(x',y)\alpha(dy|x)\pi(d(x',x)). 
	\end{equation} Notice that if $m,m'$ and $\alpha$ have the finite second moments, then $\pi*\alpha$ has also a finite second moment.
	\item We assume that the state space for each agent is the $d$-dimensional torus $\td\triangleq\rd/\mathbb{Z}^d$. Elements of $\td$ are sets $x=\{x'\}\subset \rd$ such that if $x',x''\in x$, then $x'-x''\in \mathbb{Z}^d$. 
	
	\item The distance on $\td$ is introduced as follows: for $x,y\in\td$, set
	$$\|x-y\|\triangleq \inf\{\|x'-y'\|:x'\in x,y'\in y\}. $$
	\item If $x\in\td$, $v\in\rd$, then $x+v$ is the set $\{x'+v:x'\in x\}$.
	\item $\mathbb{B}_c$ denotes  the ball in $\rd$ of the radius $c$  centered at the origin.
	\item For $s,r\in \mathbb{R}$, $s<r$, $\mathcal{C}_{s,r}$ stands for the set of continuous functions $x(\cdot)$ from $[s,r]$ to $\td$.
	\item If $x(\cdot)\in \mathcal{C}_{s,r}$, $t\in [s,r]$, then $e_t(x(\cdot))\triangleq x(t)$. With some abuse of notation, we denote the distance between $x(\cdot),y(\cdot)\in \mathcal{C}_{s,r}$ by $\|x(\cdot)-y(\cdot)\|$. Recall that 
	$$ \|x(\cdot)-y(\cdot)\|\triangleq \max_{t\in [s,r]}\|x(t)-y(t)\|.$$
	
\end{itemize}

\subsection{Mean field type differential game}
We consider the mean field type differential game with the dynamics of each agent given by
\begin{equation}\begin{split}
\frac{d}{dt}x(t)=f(&t,x(t),m(t),u(t),v(t)),\\ &t\in [0,T],\ \ m(t)\in\mathcal{P}^2(\td),\ \ u(t)\in U,\ \ v(t)\in V.\end{split}
\end{equation} Here $m(t)$ is the distribution of all agents at time $t$; $u(t)$ (respectively, $v(t)$) is the control of the first (respectively, second) player acting on the agent; $U$ (respectively, $V$) is the control space for the first (respectively, second) player.

We assume that the players influence upon the dynamics of each player independently. The purpose of the first (respectively, second) player is to minimize (respectively, maximize) the functional
$$g(m(T)). $$

We impose the following condition on the dynamics and the payoff function:
\begin{enumerate}
	\item the sets $U$ and $V$ are metric compacts;
	\item the functions $f$ and $g$ are continuous;
	\item the function $f$ is Lipschitz continuous w.r.t. $x$ and $m$ i.e. there exists $L>0$ such that, for any $t\in [0,T]$, $x',x''\in\td$, $m',m''\in\mathcal{P}(\td)$, $u\in U$, $v\in V$,
	$$\|f(t,x',m',u,v)-f(t,x',m',u,v)\|\leq L(\|x'-x''\|+W_2(m',m'')); $$
	\item (Isaacs' condition) for any $t\in [0,T]$, $x\in\td$, $m\in\ptd$, $u\in U$, $v\in V$ and $w\in\rd$,
	$$\min_{u\in U}\max_{v\in V}\langle w,f(t,x,m,u,v)\rangle=\max_{v\in V}\min_{u\in U}\langle w,f(t,x,m,u,v)\rangle. $$ 
\end{enumerate}

Now, let us describe the dynamics of the representative agent produced by the players' controls. Denote by $\mathcal{U}^0$ the set of measurable functions defined on $[0,T]$ with values in $U$. Further, set $\mathcal{U}\triangleq \Lambda([0,T],\lambda,U)$, where $\lambda$ stands for the Lebesgue measure. Below we regard $U$, $\mathcal{U}^0$ and $\mathcal{U}$ as the sets of constant, measurable and relaxed controls of the first players respectively. Using the embedding provided by the Dirac measures one can assume that
$$U\subset\mathcal{U}^0\subset\mathcal{U}. $$

Analogously, we define the set of measurable controls of the second player $\mathcal{V}^0\triangleq \{v(\cdot):[0,T]\rightarrow V\text{ measurable}\}$ and the set of relaxed control of the second player $\mathcal{V}\triangleq \Lambda([0,T],\lambda,V)$. As above, we have that
$$V\subset \mathcal{V}^0\subset\mathcal{V}. $$

Now, assume that $s,r\in [0,T]$, $y\in\td$ is an initial position, $m(\cdot):[s,r]\rightarrow\mathcal{P}^2(\td)$ is a given flow of probabilities, $\xi\in\mathcal{U}$, $\zeta\in\mathcal{V}$ are relaxed controls of the first and second player respectively. The corresponding motion of representative agent is a function defined on $[s,r]$ with values in $\td$ solving the following initial value problem:
\begin{equation}\label{motion:eq:agent}
\frac{d}{dt}x(t)=\int_U\int_V f(t,x(t),m(t),u,v)\xi(du|t)\zeta(dv|t),\ \ x(s)=y.
\end{equation} Below we denote the solution of (\ref{motion:eq:agent}) by $x(\cdot,s,y,m(\cdot),\xi,\zeta)$. Furthermore, let $\mathrm{traj}^{s,r}_{m(\cdot)}$ stand for the operator which assigns to the triple $(y,\xi,\zeta)$ the  trajectory $x(\cdot,s,y,m(\cdot),\xi,\zeta)$.

Now, let us turn to the dynamics of the whole system. We start with the analogs of the open-loop controls. It is natural to assume that at each point $x$ the player can share his/her control. Thus, we consider the distributions of controls. If $m_*\in\mathcal{P}^2(\td)$ is an initial distribution of agents, let \begin{itemize}
	\item $\mathcal{A}^c[m_*]\triangleq \Lambda(\td,m_*,U)$ be the set of distributions of constant controls of the first player;
	\item $\mathcal{A}^0[m_*]\triangleq\Lambda(\td,m_*,\mathcal{U}^0)$ denote the set of distributions of measurable controls of the first player;
	\item $\mathcal{A}[m_*]\triangleq\Lambda(\td,m_*,\mathcal{U})$ be the set of distributions of relaxed controls of the first player.
\end{itemize} Without loss of generality, we assume that
$$\mathcal{A}^c\subset\mathcal{A}^0\subset\mathcal{A}. $$

Analogously, $\mathcal{B}^c[m_*]\triangleq\Lambda(\td,m_*,V)$, $\mathcal{B}^0[m_*]\triangleq\Lambda(\td,m_*,\mathcal{V}^0)$, $\mathcal{B}[m_*]\triangleq\Lambda(\td,m_*,\mathcal{V})$ are the sets of distributions of constant, measurable and relaxed controls of the second player respectively. As above, we have that
$$\mathcal{B}^c\subset\mathcal{B}^0\subset\mathcal{B}. $$

Set $\mathcal{D}[m_*]\triangleq \Lambda(\td,m_*,\mathcal{U}\times\mathcal{V})$. In particular, this means that $\mathcal{D}[m_*]\subset\mathcal{P}(\td\times \mathcal{U}\times\mathcal{V})$ and each element of  $\mathcal{D}[m_*]$ is a distribution of pairs of controls. Finally, let us introduce the  consistent distributions of controls. If $\alpha\in\mathcal{A}[m_*]$, then let $\mathcal{D}_1[\alpha]$ be the set of probabilities $\varkappa\in\mathcal{D}[m_*]$ such that its marginal distribution on $\td\times \mathcal{U}$ is equal to $\alpha$ i.e. $\mathrm{p}^{1,2}_\#\varkappa=\alpha$. Analogously, for $\beta\in\mathcal{B}[m_*]$, we denote by $\mathcal{D}_2[\beta]$ the set of probabilities $\varkappa\in\mathcal{D}[m_*]$ such that   $\mathrm{p}^{1,3}_\#\varkappa=\beta$.

In  the following, we call any function of time taking values in the space of probabilities a flow of probabilities.

\begin{definition}\label{def:motion} Let $s,r\in [0,T]$, $s<r$, $m_*\in\mathcal{P}^2(\td)$, $\varkappa\in\mathcal{D}[m_*]$ we say that the flow of probabilities $[s,r]\ni t\mapsto m(t)\in\mathcal{P}(\td)$ is produced by $s$, $m_*$ and distribution of pairs of controls $\varkappa$ if there exists $\chi\in\mathcal{P}_2(\mathcal{C}_{s,r})$ such that
	\begin{enumerate}[label=(\roman*)]
		\item $m(t)=e_t{}_\#\chi$, $m(s)=m_*$;
		\item $\chi=\mathrm{traj}^{s,r}_{m(\cdot)}{}_\#\varkappa$.
	\end{enumerate}
\end{definition} Below we denote the flow of probabilities produced by $s$, $m_*$ and $\varkappa$ by $m(\cdot,s,m_*,\varkappa)$.

\subsection{Feedback strategies and value function}

\begin{definition}\label{def:strategy}
	A feedback strategy of the first player is a function $\mathfrak{u}:[0,T]\times\ptd\rightarrow\mathcal{P}(\td\times U)$ satisfying the condition $$\mathfrak{u}[s,m]\in\mathcal{A}^c[m]$$ for each $s\in [0,T]$ and $m\in\ptd$. Analogously, we call any function $\mathfrak{v}:[0,T]\times\mathcal{P}(\td)\rightarrow\mathcal{P}(\td\times V)$ such that, for any $s\in [0,T]$, $m\in\mathcal{P}(\td)$, $\mathfrak{v}[s,m]\in\mathcal{D}^c[m_*]$ a feedback strategy of the second player.
\end{definition}

We assume that the players form their controls stepwise. If $t_0\in [0,T]$ is an initial time, $m_0\in\ptd$ is an initial distribution of agents, $\mathfrak{u}$ is a strategy of the first player, $\Delta=\{t_k\}_{k=0}^N$ is a partition of $[t_0,T]$, then we say that the flow of probabilities $[t_0,T]\ni t\mapsto m(t)$ is produced by $t_0$, $m_0$, $\mathfrak{u}$ and $\Delta$ if there exist  sequences of probabilities $\{\alpha_k\}\subset\mathcal{P}(\td\times U)$ $\{\varkappa_k\}_{k=0}^N\subset\mathcal{P}(\td\times U\times \mathcal{V})$ such that
\begin{enumerate}
	\item $\alpha_k=\mathfrak{u}[t_k,m_k]$, where $m_k\triangleq m(t_k)$;
	\item $\varkappa_k\in\mathcal{D}_1[\alpha_k]$;
	\item for $k=0,\ldots,N-1$, $t\in [t_{k},t_{k+1}]$, $$m(t)=m(t,t_k,m_k,\varkappa_k). $$
\end{enumerate} Notice that $t_k$ is the time of control correction.

We denote the set of flows of probabilities produced by $t_0$, $m_0$, $\mathfrak{u}$ and $\Delta$ by $\mathcal{X}_1(t_0,m_0,\mathfrak{u},\Delta)$. 

Given an initial position $(t_0,m_0)\in[0,T]\times\ptd$, a second player's control $\mathfrak{v}$ and a sequence of times of control correction $\Delta=\{t_k\}_{k=0}^N$, one can define the corresponding set of flows of probabilities in the similar way. We denote it by $\mathcal{X}_2(t_0,m_0,\mathfrak{v},\Delta)$. 

The the players' outcome can be estimated as follows:
$$\Gamma_1(t_0,m_0)\triangleq \inf_{\mathfrak{u},\Delta}\sup_{m(\cdot)\in\mathcal{X}_1(t_0,m_0,\mathfrak{u},\Delta)}g(m(T));$$
$$\Gamma_2(t_0,m_0)\triangleq \sup_{\mathfrak{v},\Delta}\inf_{m(\cdot)\in\mathcal{X}_2(t_0,m_0,\mathfrak{v},\Delta)}g(m(T)).$$ The function $\Gamma_1$ (respectively, $\Gamma_2$) is the upper (respectively, lower) value of the game. Under the imposed condition it is proved (see~\cite{Averbokh_mfdg}) that there exists a value function of the game $\Gamma=\Gamma_1=\Gamma_2$. To characterize the value function we need the notions of $u$- and $v$-stability. 

\begin{definition}\label{def:u_stable} We say that a lower semicontinuous function $\psi_1:[0,T]\times\ptd\rightarrow\mathbb{R}$ is $u$-stable if
	\begin{itemize}
		\item for any $m\in\ptd$, $g(m)\geq \psi_1(T,m)$;
		\item for any $s,r\in[0,T]$, $s\leq r$, $m_*\in\ptd$, $\beta\in\mathcal{B}^c[m_*]$, there exists $\varkappa\in \mathcal{D}_2[\beta]$ such that 
		$$\psi_1(s,m_*)\geq \psi_1(r,m(r,s,m_*,\varkappa)).$$
	\end{itemize}
\end{definition}

\begin{definition}\label{def:v_stable} An upper semicontinuous function $\psi_2:[0,T]\times\ptd\rightarrow\mathbb{R}$ is said to be  $v$-stable if
	\begin{itemize}
		\item for any $m\in\ptd$, $g(m)\leq \psi_2(T,m)$;
		\item for any $s,r\in[0,T]$, $s\leq r$, $m_*\in\ptd$, $\alpha\in\mathcal{A}^c[m_*]$, there exists $\varkappa\in \mathcal{D}_1[\alpha]$ such that 
		$$\psi_2(s,m_*)\leq \psi_2(r,m(r,s,m_*,\varkappa)).$$
	\end{itemize}
\end{definition}

\begin{theorem}[Theorems 1, 2~\cite{Averbokh_mfdg}]\label{th:u_v_stability_mfdg}
	If $\psi_1$ is $u$-stable, then $\Gamma_1\leq \psi_1$. Analogously, if $\psi_2$ is $v$-stable, then $\Gamma_2\geq \psi_2$.
	
	The value function exists and it is simultaneously $u$- and $v$-stable.
\end{theorem}
\begin{remark}\label{remark:strategy} Given a $u$-stable (respectively, $v$-stable) function one can construct the suboptimal strategy of the first (respectively, second) player. This is based on the variant of the extremal shift rule for the mean field type differential games (see~\cite{Averbokh_mfdg} for details).
\end{remark}

\section{Main result}\label{sect:main_result} In this section we extend the notion of directional derivatives to functions of probability and formulate the infinitesimal variants of   $u$- and $v$-stability conditions using this notion.

For $\tau>0$, let $\Theta^\tau:\tdu\times\rd\rightarrow\td$ be defined by the rule:
$$\Theta^\tau(x,u,w)\triangleq x+\tau w. $$ With some abuse of notation, we also denote by $\Theta^\tau$ the operator from $\tdv\times\rd$ to $\td$ acting by the rule:
$$\Theta^\tau(x,v,w)\triangleq x+\tau w. $$

\begin{definition}\label{def:directional_derivative}
	Let $\psi:[0,T]\times\ptd\rightarrow\mathbb{R}$, $s\in [0,T]$, $c>0$, $\eta\in\mathcal{P}^2(\tdv\times\rd), $ $\beta\triangleq \mathrm{p}^{1,2}{}_\#\eta$, $m\triangleq \mathrm{p}^1{}_\#\beta=\mathrm{p}^1{}_\#\eta$. The value
	$$\ud_c\psi(s;\eta)\triangleq \liminf_{\substack{\eta'\in\ptdvbc,\ \ \mathrm{p}^{1,2}{}_\#\eta'=\beta\\ \tau\downarrow 0,\ \ W_2(\eta',\eta)\downarrow 0 }}\frac{\psi(s+\tau,\Theta^{\tau}{}_\#\eta')-\psi(s,m)}{\tau}$$ is called a $u$-derivative of the function $\psi$ at $s$ and $\eta$ for the radius $c$. 
	
	Analogously, if $\eta\in \mathcal{P}^2(\tdu\times\rd)$, $\alpha\triangleq\mathrm{p}^{1,2}{}_\#\eta$, $m\triangleq\mathrm{p}^1{}_\#\alpha=\mathrm{p}^1{}_\#\eta$, $c>0$, the number
	$$\vd_c\psi(s;\eta)\triangleq \limsup_{\substack{\eta'\in\ptdubc,\ \ \mathrm{p}^{1,2}_\#\eta'=\alpha\\ \tau\downarrow 0,\ \ W_2(\eta',\eta)\downarrow 0 }}\frac{\psi(s+\tau,\Theta^{\tau}{}_\#\eta')-\psi(s,m)}{\tau}$$ is  a $v$-derivative of the function $\psi$ at $s$ and $\eta$  for the radius $c$.
\end{definition}

Notice that one can consider $u$-derivative  as a lower directional derivative of the extension of the function $\psi$ to the space $[0,T]\times\mathcal{P}^2(\tdv)$. Analogously, $v$-derivative can be regarded to be an upper directional derivative of lifting of $\psi$ to $[0,T]\times\mathcal{P}^2(\tdv)$. However, since there is no natural ways to define the set of tangent distribution to $\mathcal{P}^2(\tdu)$ (respectively, $\mathcal{P}^2(\tdv)$) we use only shifts on $\td$ and introduce the special notions.

Now, we define analogs of the vectograms. 
First, for  $s\in [0,T]$, $x\in\td$, $m\in\ptd$, $u\in U$, $v\in V$, put $$F_1(s,x,m,u)\triangleq \mathrm{co}\{f(t,x,m,u,v):v\in V\}, $$
$$F_2(s,x,m,v)\triangleq \mathrm{co}\{f(s,x,m,u,v):u\in U\}. $$ The graphs of   $F_1$ and $F_2$ are introduced as follows. Set
$$G_1(s,m)\triangleq \{(x,u,w):x\in\td,u\in U,w\in F_1(s,x,m,u)\}, $$
$$G_2(s,m)\triangleq \{(x,v,w):x\in\td,v\in V,w\in F_2(s,x,m,v)\}. $$ Now, let $s\in[0,T]$,
$\alpha\in \mathcal{P}(\tdu)$, $m\triangleq \mathrm{p}^1{}_\#\alpha$. Put $$\mathcal{F}_1(s,\alpha)\triangleq\{\eta:\eta\in\mathcal{P}^2(G_1(s,m)),\ \ \mathrm{p}^{1,2}{}_\#\eta=\alpha \}.$$ If $\beta\in\mathcal{P}(\tdv)$, $m\triangleq \mathrm{p}^1{}_\#\beta$, then set
$$\mathcal{F}_2(s,\beta)\triangleq\{\eta:\eta\in\mathcal{P}^2(G_2(s,m)),\ \ \mathrm{p}^{1,2}{}_\#\eta=\beta \}.$$

The set $\mathcal{F}_1(s,\alpha)$ (respectively, $\mathcal{F}_2(s,\beta)$) plays the role the vectogram for the given distribution of the constant controls of the first (respectively, second) player.

\begin{theorem}\label{th:characterisation} A lower semicontinuous function $\psi_1:[0,T]\rightarrow\ptd$ is $u$-stable if and only if
	\begin{itemize}
		\item 
		for any $m\in\ptd$, $g(m)\geq \psi_1(T,m)$;
		\item there exists $c>0$ such that, for any $s\in [0,T]$,  $\beta\in\mathcal{P}(\tdv)$,
		$$\inf\{\ud_c\psi_1(s,\eta):\eta\in\mathcal{F}_2(s,\beta)\}\leq 0. $$
	\end{itemize}	
	A upper semicontinuous function $\psi_2:[0,T]\rightarrow\ptd$ is $v$-stable if and only if
	\begin{itemize}
		\item  for any $m\in\ptd$, $g(m)\leq \psi_2(T,m)$;
		\item for any $s\in [0,T]$, $\alpha\in\mathcal{P}(\tdu)$,
		$$\sup\{\vd_c\psi_2(s,\eta):\eta\in\mathcal{F}_1(s,\alpha)\}\geq 0, $$ where $c$ is constant independent  of $s$ and $\alpha$.
	\end{itemize}
\end{theorem}

This and Theorem~\ref{th:u_v_stability_mfdg} immediately imply the following.

\begin{corollary}\label{corollary:value} A continuous function $\psi:[0,T]\times\ptd\rightarrow\mathbb{R}$ is a value function of the mean field type differential game if and only if, for any $m\in\ptd$, $g(m)=\psi(T,m)$ and one can find a constant $c>0$ satisfying the following condition: for each $s\in [0,T]$, $\alpha\in \mathcal{P}(\tdu)$, $\beta\in\mathcal{P}(\tdv)$,
	\begin{itemize}
		\item $\inf\{\ud_c\psi(s,\eta):\eta\in\mathcal{F}_2(s,\beta)\}\leq 0; $
		\item $\sup\{\vd_c\psi(s,\eta):\eta\in\mathcal{F}_1(s,\alpha)\}\geq 0.$
	\end{itemize}
\end{corollary}

\section{Flows produced by distribution of constant controls}\label{sect:equivalence}

Below we will consider only the $v$-stability condition. The case of the $u$-stability is studied in the similar way.

First, we replace the metric on $U$. Originally, we consider on $U$ a metric $\rho_U$. Now, let $\varpi$ be a modulus of continuity for $f$. In particular, for $t\in [0,T]$, $x\in\td$, $m\in\ptd$, $u',u''\in U$, $v\in V$,
\begin{equation*}\|f(t,x,m,u',v)-f(t,x,m,u'',v)\|\leq \varpi(\rho_U(u',u'')).\end{equation*} Put
\begin{equation}\label{equivalence:intro:hat_rho}
\hat{\rho}_U(u',u'')\triangleq \varpi(\rho_U(u',u''))+\rho_U(u',u''). 
\end{equation} Obviously, $\hat{\rho}_U$ is a metric on $U$ and the function $f$ is Lipschitz continuous w.r.t. $u$ in $(U,\hat{\rho}_U)$. However, we are to prove that the set of Borel probabilities does not change when we change the metric. To this end we prove the following.
\begin{proposition}\label{prop:topology_U} The topologies produced by $\rho_U$ and $\hat{\rho}_U$ coincides.
\end{proposition}
\begin{proof}
	Assume that $E$ is open within $\rho_U$. Let $u\in E$ and let $\varepsilon>$ be such that $ \{u'\in U:\rho_U(u,u')<\varepsilon\}\subset E$. Since $\rho_U(u,u')\leq \hat{\rho}_U(u,u')$, we have that $$\{u'\in U:\hat{\rho}_U(u,u')<\varepsilon\}\subset \{u'\in U:\rho_U(u,u')<\varepsilon\}\subset E. $$ 
	
	Conversely, assume that $E$ is open within $\hat{\rho}_U$. Pick any $u\in E$. There exists $\varepsilon>0$ such that 
	$$\{u'\in U:\hat{\rho}_U(u,u')<\varepsilon\}\subset E. $$ Since $\varpi(\delta)+\delta\rightarrow 0$ as $\delta\rightarrow 0$, we have that there exists $\delta$ such that, for any $\delta'<\delta$, $\varpi(\delta')+\delta'<\varepsilon$. Using (\ref{equivalence:intro:hat_rho}), we get that if ${\rho}_U(u,u')<\delta$, then $\hat{\rho}_U(u,u')<\varepsilon$. Thus,
	$$\{u'\in U:{\rho}_U(u,u')<\delta\}\subset\{u'\in U:\hat{\rho}_U(u,u')<\varepsilon\}\subset E. $$
\end{proof}

Now, let us rewrite the $v$-stability condition using the probabilities on the set $\ctu{s}{r}$. To this end we need some additional designations. For $s,r\in [0,T]$, $t\in [s,r]$, $a=(x(\cdot),u)$, put
$$\et{t}(a)\triangleq x(t),$$
$$\hae{t}(a)\triangleq (x(t),u). $$ 

\begin{proposition}\label{prop:motion_evivalence}
	Let $s,r\in [0,T]$, $s<r$, $m_*\in\ptd$, $\alpha\in\mathcal{A}^c[m_*]$, a flow of probabilities $[s,r]\ni t\mapsto m(t)$ is a motion produced by $s$, $m_*$ and some $\varkappa\in\mathcal{D}_1[\alpha]$ iff there exists $\nu\in \pctu{s}{r})$ such that
	\begin{enumerate}
		\item $\hae{s}{}_\#\nu=\alpha$;
		\item $\et{t}{}_\#\nu=m(t)$ for every $t\in [0,T]$;
		\item for any $t',t''\in [0,T]$,
		$$\int_{\ctu{s}{r}}\mathrm{dist}\left(\et{t''}(a)-\et{t'}(a),\int_{t'}^{t''} F_1(t,\et{t}(a),m(t),\mathrm{p}^2(a))dt\right)\nu(da)=0; $$
	\end{enumerate}
\end{proposition}
\begin{proof}
	First, assume that $m(\cdot)=m(\cdot,s,m_*,\varkappa)$ for some $\varkappa\in\mathcal{D}_1[\alpha]$. This means that there exists $\chi\in\mathcal{P}(\mathcal{C}_{0,T})$ such that $m_*=m(s)=e_s{}_\#\chi$ and $\chi=\mathrm{traj}_{m(\cdot)}^{s,r}{}_\#\varkappa.$ Note that if $(x,u,\zeta)\in \td\times U\times \mathcal{V}$, then $x(\cdot)=\mathrm{traj}^{s,r}_{m(\cdot)}(x,u,\zeta)$ satisfies the differential inclusion
	\begin{equation}\label{equivalence:incl:F_1}
	\frac{d}{dt}x(t)\in F_1(t,x(t),m(t) ,u).
	\end{equation} This implies that, for any $t',t''\in [s,r]$,
	\begin{equation}\label{equivalence:equality:int_F_1}
	\mathrm{dist}\left(x(t'')-x(t'),\int_{t'}^{t''}F_1(t,x(t),m(t),u)\right)=0. 
	\end{equation} Introduce the mapping $\mathrm{Traj}^{s,r}_{m(\cdot)}:\td\times U\times\mathcal{V}\rightarrow \ctu{s}{r}$ by the rule
	$$\mathrm{Traj}^{s,r}_{m(\cdot)}(x,u,\zeta)\triangleq (\mathrm{traj}^{s,r}_{m(\cdot)}(x,u,\zeta),u) $$ and let $\nu\triangleq \mathrm{Traj}^{s,r}_{m(\cdot)}{}_\#\varkappa$. By construction we have that $\mathrm{p}^1{}_\#\nu=\chi$, $\hae{s}{}_\#\nu=\alpha$, $\et{t}{}_\#\nu=m(t)$. Finally, using (\ref{equivalence:equality:int_F_1}) we get the third condition.  
	
	Conversely, assume that there exists a probability $\nu\in\pctu{s}{r}$ such that conditions 1--3 are fulfilled. Let $\mathrm{SOL}_{m(\cdot)}^{s,r}$ denote the set of pairs $(x(\cdot),u)\in\ctu{s}{r}$ such that $x(\cdot)$ is absolutely continuous and (\ref{equivalence:incl:F_1}) holds for a.e. $t\in [s,r]$. We shall prove that $\nu$ is concentrated on $\mathrm{SOL}_{m(\cdot)}^{s,r}$. Indeed, there exists a $\nu$-null set $\mathcal{N}\subset\ctu{s}{r}$ such that, for any $(x(\cdot),u)\in(\ctu{s}{r})\setminus \mathcal{N}$ and any rational $t'$, $t''$,
	\begin{equation}\label{equivalence:equality_F_1_t_prime}
	x(t'')-x(t')=\int_{t'}^{t''}F_1(t,x(t),m(t),u). 
	\end{equation} Passing to the limit, we get that (\ref{equivalence:equality_F_1_t_prime}) for every $(x(\cdot),u)\in(\ctu{s}{r})\setminus\mathcal{N}$ and all $t',t''\in [s,r]$. Therefore, $(\ctu{s}{r})\setminus\mathcal{N}\subset \mathrm{SOL}_{m(\cdot)}^{s,r}$. Since $\nu(\mathcal{N})=0$, we can assume that $\nu$  is itself concentrated on $\mathrm{SOL}_{m(\cdot)}^{s,r}$. 
	
	Further, for $(x(\cdot),u)\in\mathrm{SOL}_{m(\cdot)}^{s,r}$, let $\mathbf{b}(x(\cdot),u)$ be the set of triples $(y,u,\zeta)\subset \tdu\times\mathcal{V}$ such that $y=x(s)$ and $x(\cdot)=\mathrm{traj}_{m(\cdot)}^{s,r}(y,u,\zeta)$. By the continuity of $\mathrm{traj}$ we have that $\mathbf{b}$ is upper semicontinuous. Moreover, $\mathbf{b}(x(\cdot),u)$ is nonempty when $(x(\cdot),u)\in\mathrm{SOL}_{m(\cdot)}^{s,r}$. Thus, by the Kuratowski–Ryll-Nardzewski selection theorem  \cite[Theorem 18.13]{Infinite_dimensional_analysis} there exists a measurable function $\mathbf{b}_0:\mathrm{SOL}_{m(\cdot)}^{s,r}\rightarrow\tdu\times\mathcal{V}$ such that $\mathbf{b}_0(x(\cdot),u)\in\mathbf{b}(x(\cdot),u)$. Put $\varkappa\triangleq \mathbf{b}_0{}_\#\nu$. By construction we have that
	$$\mathrm{p}^{1,2}{}_\#\varkappa=\hae{s}{}_\#\nu=\alpha. $$ This means that $\varkappa\in\mathcal{D}_1[\alpha]$. Further, put $$\chi\triangleq\mathrm{traj}^{s,r}_{m(\cdot)}{}_\#\varkappa= \mathrm{p}^1{}_\#\nu. $$ Therefore, $m(t)=e_t{}_\#\chi$. Since $m(s)=m_*$, we get that $m(\cdot)$ is produced by $s$, $m_*$ and $\varkappa\in\mathcal{D}_1[\alpha]$.
\end{proof}

Proposition~\ref{prop:motion_evivalence} and the definition of the $v$-stability (see Definition~\ref{def:v_stable}) imply the following.

\begin{corollary}\label{corollary:equivalence} A upper semicontinuous functions $\psi_2:[0,T]\times\ptd$ is $v$-stable iff $g(m)\leq \psi_2(T,m)$ and, given $s,r\in [0,T]$, $\alpha_*\in\mathcal{P}(\tdu)$, there exists $\nu\in\pctu{s}{r}$  such that
	\begin{enumerate}
		\item $e_s{}_\#\nu=\alpha_*$;
		\item for any $t',t''\in [s,r]$,
		$$\int_{\pctu{s}{r}}\mathrm{dist}\left(\et{t''}(a)-\et{t'}(a),\int_{t'}^{t''} F_1(t,\et{t}(a),\et{t}{}_\#\nu,\mathrm{p}^2(a))dt\right)\nu(d(a(\cdot)))=0; $$
		\item $\psi_2(s,\et{s}{}_\#\nu)\leq \psi_2(r,\et{r}{}_\#\nu).$
	\end{enumerate}
\end{corollary}

\section{Properties of the shift operator}\label{sect:shift_operator}
Given $\tau>0$, define the operator $\Xi^\tau:\tdu\times\rd\rightarrow\tdu$ by the rule:
$$\Xi^\tau(x,u,w)\triangleq (x+\tau w,u). $$ Notice that the operator $\Xi^\tau$ can be regarded as an extension of the operator $\Theta^\tau$ defined above. This means that
\begin{equation}\label{shift:equality:Xi_Theta}
\mathrm{p}^1(\Xi^\tau(x,u,w))=\Theta^\tau(x,w).
\end{equation}

The following lemmas are concerned with the transfer of distribution of direction determined by the composition operation  $*$ defined by (\ref{notation:intro:composition}). 

\begin{lemma}\label{lm:continuity_Xi}
	Let $c$ be a positive number, $\alpha,\alpha'\in\mathcal{P}^2(\tdu)$, $\eta\in \mathcal{P}^2(\tdubc)$ be such that $\mathrm{p}^{1,2}{}_\#\eta=\alpha$, $\tau,\theta>0$, $\pi\in \Pi^0(\alpha',\alpha)$ be an optimal plan between $\alpha'$ and $\alpha$, then
	$$W_2(\Xi^\tau{}_\#\eta,\Xi^{\theta}{}_\#(\pi*\eta))\leq W_2(\alpha,\alpha')+|\tau-\theta| c. $$
\end{lemma}
\begin{proof}
	Let us consider the plan $\hat{\pi}$ between $\Xi^{\theta}{}_\#(\pi*\eta)$ and $\Xi^{\tau}{}_\#\eta$ given by the rule: for $\phi\in C_b(\tdu)$,
	\begin{equation*}
	\begin{split}
	\int_{\tdu\times \tdu}\phi&(y',u',y,u)\hat{\pi}(d(y',u',y,u))\\ &\triangleq
	\int_{\tdu\times\tdu}\int_{\mathbb{B}_c}\phi(x'+\theta w,u',x+\tau w,u)\eta(dw|(x,u)\pi(d(x',x)).
	\end{split}
	\end{equation*}
	Hence, using the Minkowski inequality for the functions $\phi',\phi'':\tdu\times\tdu\times\rd\rightarrow\mathbb{R}^{d+1}$, defined by the rule:
	$$\phi'(y',u',y,u,w)\triangleq (y'-y,\hat{\rho}_U(u',u)),\ \ \phi''(y',u',y,u,w)\triangleq  ((\tau-\theta)w,0),$$
	we conclude that
	\begin{equation*}\begin{split}
	W_2(&\Xi^{\theta}{}_\#(\pi*\eta),\Xi^\tau{}_\#\eta)\\&\leq \Bigl[\int_{\tdu\times\tdu} (\|y'-y\|^2+\hat{\rho}^2_U(u',u))\hat{\pi}(d(y',u',y,u))\Bigr]^{1/2}\\ &=
	\Bigl[\int_{\tdu\times\tdu}\int_{\mathbb{B}_c} (\|x'-x+(\theta-\tau)w\|^2+\hat{\rho}^2_U(u',u))\eta(dw|(x,u))\pi(d(x',u',x,u))\Bigr]^{1/2} \\
	&=\Bigl[\int_{\tdu\times\tdu}\int_{\mathbb{B}_c}\|\phi'(y',u',y,u,w)+\phi''(y',u',y,u,w)\|^2\eta(dw|(x,u))\pi(d(x',u',x,u))\Bigr]^{1/2}\\ 
	&\leq \Bigl[\int_{\tdu\times\tdu}\int_{\mathbb{B}_c}\|\phi'(y',u',y,u,w)\|^2\eta(dw|(x,u))\pi(d(x',u',x,u))\Bigr]^{1/2}
	\\ &{}\hspace{15pt}+\Bigl[\int_{\tdu\times\tdu}\int_{\mathbb{B}_c}\|\phi''(y',u',y,u,w)\|^2\eta(dw|(x,u))\pi(d(x',u',x,u))\Bigr]^{1/2}\\
	 &=
	\Bigl[\int_{\tdu\times\tdu}\int_{\mathbb{B}_c} (\|x'-x\|^2+\hat{\rho}^2_U(u',u))\eta(dw|(x,u))\pi(d(x',u',x,u))\Bigr]^{1/2} \\ &{}\hspace{15pt}+
	\Bigl[\int_{\tdu\times\tdu}\int_{\mathbb{B}_c} |\theta-\tau|^2\|w\|^2\eta(dw|(x,u))\pi(d(x',u',x,u))\Bigr]^{1/2}. 
	\end{split}
	\end{equation*} This implies the conclusion of the lemma.
\end{proof}

Recall that $L$ denotes the Lipschitz constant for the function $f$ w.r.t. $x$ and $m$, whereas $\varpi$ is the modulus of continuity of $f$ w.r.t. $t$, $u$ and $v$. Assuming that $L\geq 1$ and using the definition of the metric $\hat{\rho}_U$ (see~(\ref{equivalence:intro:hat_rho})), we have that
\begin{equation}\label{shift:ineq:f}
\begin{split}
\|f(t,x',m',u',v)&-f(t,x'',m'',u'',v)\|\\ &\leq \varpi(t'-t'')+L(\|x'-x''\|+W_2(m',m'')+\hat{\rho}_U(u',u'')).
\end{split}
\end{equation}

\begin{lemma}\label{lm:lipschtitz_F}
	Let $\alpha,\alpha'\in\mathcal{P}(\tdu)$, $m\triangleq\mathrm{p}^1{}_\#\alpha$, $m'\triangleq\mathrm{p}^1{}_\#\alpha'$, $\eta\in \ptdubc$ be such that $\mathrm{p}^{1,2}{}_\#\eta=\alpha$, $\pi\in \Pi^0(\alpha',\alpha)$ be an optimal plan between $\alpha'$ and $\alpha$, $\eta'\triangleq \pi*\eta$. Then
	\begin{equation*}
	\begin{split}
	\Bigl|\int_{\tdubc}\mathrm{dist}(w,F_1(t,x,m,u))\eta(d(x,u,w)&)\\- \int_{\tdubc}\mathrm{dist}(w,F_1(t',x,m',u&))\eta'(d(x,u,w))\Bigr|  \\&\leq \varpi(t'-t)+ 2LW_2(\alpha',\alpha).
	\end{split} 
	\end{equation*}
\end{lemma}
\begin{proof}
	
	By definition of $\eta'$ we have that
	\begin{equation}\label{shift:estima:dist}
	\begin{split}
	\Bigl|\int_{\tdubc}\mathrm{dist}(w,F_1(t,x,m,u))\eta(d(x,u,w)&)\\- \int_{\tdubc}\mathrm{dist}(w,F_1(t',x,m',u&))\eta'(d(x,u,w))\Bigr| \\ \leq
	\int_{\tdu\times\tdu}\int_{\mathbb{B}_c}\Bigr|\mathrm{dist}(w,F_1(t,x,m,u)&)- \mathrm{dist}(w,F_1(t',x',m',u'))\Bigr|\\ 
	&{}\hspace{28pt}\eta(dw|x,u)\pi(d(x',u',x,u)).
	\end{split} 
	\end{equation}	
	To estimate $|\mathrm{dist}(w,F_1(t,x,m,u))- \mathrm{dist}(w,F_1(t',x',m',u'))|$ recall that
	$$F_1(t,x,m,u)=\left\{\int_V f(t,x,m,u,v)\omega_v(dv):\omega_v\in\mathcal{P}(V) \right\}. $$ Pick $\omega_v^*$ such that
	\begin{equation*}
	\begin{split}
	\mathrm{dist}(w,F_1(t,x,m,u))=\min_{\omega_v\in\mathcal{P}(V)}&{}\Bigl\|w-\int_V f(t,x,m,u,v)\omega_v(dv)\Bigr\|\\=&\Bigl\|w-\int_V f(t,x,m,u,v)\omega_v^*(dv)\Bigr\|.
	\end{split} 
	\end{equation*} We have that
	$$
	\mathrm{dist}(w,F_1(t',x',m',u'))\leq \Bigl\|w-\int_V f(t',x',m',u',v)\omega_v^*(dv)\Bigr\|. $$
	Thus, by (\ref{shift:ineq:f}) we obtain that
	\begin{equation*}
	\begin{split}
	\mathrm{dist}(w,F_1&(t,x,m,u))- \mathrm{dist}(w,F_1(t',x',m',u')) \\ &\leq
	\Bigl\|w-\int_V f(t,x,m,u,v)\omega_v^*(dv)\Bigr\|-\Bigl\|w-\int_V f(t',x',m',u',v)\omega_v^*(dv)\Bigr\| \\ &\leq \varpi(t'-t)+ L(\|x'-x\|+W_2(m',m)+\hat{\rho}_U(u',u)).
	\end{split}
	\end{equation*} The opposite inequality is established in the same way. Hence, we get the estimate
	\begin{equation*}
	\begin{split}
	|\mathrm{dist}(w,F_1(t,x,m,&u))- \mathrm{dist}(w,F_1(t',x',m',u'))|\\ &\leq \varpi(t'-t)+ L(\|x'-x\|+W_2(m',m)+\hat{\rho}_U(u',u)).
	\end{split} 
	\end{equation*} This, (\ref{shift:estima:dist}) and the Jensen's inequality yield that
	\begin{equation*}
	\begin{split}
	\Bigl|\int_{\tdubc}\mathrm{dist}(w,F_1(t,x,m,u))\eta(d(x,u,w)&)\\- \int_{\tdubc}\mathrm{dist}(w,F_1(t',x,m',u&))\eta'(d(x,u,w))\Bigr| \\ \leq L \int_{\tdu\times\tdu} (\|x'-x\|+\hat{\rho}_U(u',u))&\pi(d(x',u',x,u))\\+\varpi(t'-&t)+LW_2(m',m) \\ \leq
	L\Bigl[\int_{\tdu\times\tdu} (\|x'-x\|^2+\hat{\rho}_U(u',u&))^2\pi(d(x',u',x,u))\Bigr]^{1/2}\\+\varpi(t'-&t)+LW_2(\alpha',\alpha).
	\end{split}
	\end{equation*}
	Since $\pi$ is an optimal plan between $\alpha'$ and $\alpha$ we get that the right-hand side of this inequality is equal to $\varpi(t'-t)+2LW_2(\alpha',\alpha)$. This gives the conclusion of the lemma.
\end{proof}

\section{Proof of Theorem~\ref{th:characterisation}. Sufficiency}\label{sect:proof:sufficiency}

In this section we assume that the upper semicontinuous function $\psi_2:[0,T]\times\td\rightarrow\mathbb{R}$ is such that 
\begin{enumerate}[label=(\roman*)]
\item for every $m\in\ptd$, $\psi_2(T,m)\geq g(m)$, 
\item for every $t\in [0,T]$ and $\alpha\in \mathcal{P}(\tdu)$, 
$$\sup\{\vd_c\psi_2(t,\eta):\eta\in\mathcal{F}_1(s,\alpha)\}\geq 0, $$ where $c$ is constant independent  of $t$ and $\alpha$.
\end{enumerate} We aims to prove that in this case $\psi_2$ is $v$-stable. To this end given $s,r\in [0,T]$, $s<r$, $\alpha_*\in\mathcal{P}(\tdu)$, we construct $\nu\in\pctu{s}{r}$ satisfying conditions 1--4 of Corollary~\ref{corollary:equivalence}. 

Put $m_*\triangleq \mathrm{p}^1{}_\#\alpha_*$, $z^*\triangleq \psi_2(s,m_*)$. Let natural $n$ be such that $r-1/n>s$.

The proof of sufficiency part is based on the following.  

\begin{lemma}\label{lm:one_step} There exists $\varepsilon_n\in (0,1/n]$ satisfying the following property. For any $(t,\alpha)$ such that $t\in [s,r-1/n]$, $\psi_2(t,\mathrm{p}^1{}_\#\alpha)\geq z^*-(t-s)$, one can find $t^+\in (t+\varepsilon_n,t+1/n)$, $\alpha^+\in \ptdu$ and $\eta\in\ptdubc$ such that
	\begin{enumerate}
		\item $\psi_2(t^+,\mathrm{p}^1{}_\#\alpha^+)> \psi_2(t,\mathrm{p}^1{}_\#\alpha)-(t^+-t)/n$;
		\item $\mathrm{p}^{1,2}{}_\#\eta=\alpha$;
		\item $W_2(\Xi^{t^+-t}{}_\#\eta,\alpha^+)< (t^+-t)/n$;
		\item $$\int_{\tdubc}\mathrm{dist}(w,F_1(t,x,\mathrm{p}^1{}_\#\alpha,u))\eta(d(x,u,w))<1/n. $$
	\end{enumerate}
	
\end{lemma}
\begin{proof}
	By assumption, for any $(h,\mu)\in [s,r-1/n]\times\ptdu$, there exists $\gamma\in \ptdubc$ $\{\tau_k\},\{\epsilon_k\}\subset (0,+\infty)$, $\{\gamma_k\}\subset \ptdubc$ such that
	 $\tau_k,\epsilon_k\rightarrow 0$ as $k\rightarrow\infty$, $\mathrm{p}^{1,2}{}_\#\gamma_k=\mathrm{p}^{1,2}{}_\#\gamma=\mu$, $W_2(\gamma_k,\gamma)<\epsilon_k$, $$\psi_2(t+2\tau_k,\Theta^{2\tau_k}{}_\#\gamma_k)> \psi(t,\mathrm{p}^1{}_\#\mu)-2\tau_k\epsilon_k$$ and $$\int_{\tdubc}\mathrm{dist}(w,F_1(t,x,\mathrm{p}^1{}_\#\mu),u)\gamma(d(x,u,w))=0. $$

	Choosing $k$ sufficiently large and using Lemma~\ref{lm:lipschtitz_F}, we can find $\tau_{h,\mu}\in \{\tau_k\}$, $\gamma_{h,\mu}\in\{\gamma_k\}$ such that
	\begin{itemize}
	\item $\mathrm{p}^{1,2}{}_\#\gamma_{h,\mu}=\mu$;
	\item $3\tau_{h,\mu}<1/n$;
	\item  \begin{equation}\label{proof_suff:ineq:psi_2_h_mu}
	\psi_2(t+2\tau_{h,\mu},\Theta^{2\tau_{h,\mu}}{}_\#\gamma_{h,\mu})> \psi_2(t,\mathrm{p}^1{}_\#\mu)-\frac{2\tau_{h,\mu}}{2n};
	\end{equation}
	\item 
	\begin{equation}\label{proof_suff:ineq:F_h_mu}
	\int_{\tdubc}\mathrm{dist}(w,F_1(t,x,\mathrm{p}^1{}_\#\mu),u)\gamma_{h,\mu}(d(x,u,w))<\frac{1}{2n}.
	\end{equation} 
\end{itemize}
	Denote $h^+_{h,\mu}\triangleq h+2\tau_{h,\mu}$, $\mu^+_{h,\mu}\triangleq \Xi^{2\tau_{h,\mu}}{}_\#\gamma_{h,\mu}$.
	
	Now, let $$\mathcal{K}\triangleq \{(t,\alpha,z)\in [s,r-1/n]\times\ptdu\times \mathbb{R}:z\in [z^*-(t-s),\psi_2(t,\mathrm{p}^1{}_\#\alpha)]\}.$$ Since $\psi_2$ is upper semicontinuous, the set $\mathcal{K}$ is compact. Given $(h,\mu)\in [s,r-1/n]\times \ptdu$, let $\mathcal{E}(h,\mu)$ be the set of triples $(t,\alpha,z)\in\mathcal{K}$ such that for some $\eta\in\ptdubc$ the following inequalities are fulfilled:
	\begin{enumerate}[label=(E\arabic*)]
		\item $|t-h|<\tau_{h,\mu}$;
		\item $\psi^2(h^+_{h,\mu},\mathrm{p}^1{}_\#\mu^+_{h,\mu})>z-(h^+_{h,\mu}-t)/n$;
		\item $W_2(\Xi^{h^+_{h,\mu}-t}{}_\#\eta,\mu^+_{h,\mu})<(h^+_{t,\mu}-h)/n$;
		\item $$\int_{\tdubc}\mathrm{dist}(w,F_1(t,x,\mathrm{p}^1{}_\#\alpha),u)\eta(d(x,u,w))<\frac{1}{n}.$$ 
	\end{enumerate}
	The choice of $h^+_{h,\mu}$, $\mu^+_{h,\mu}$ and inequalities (\ref{proof_suff:ineq:psi_2_h_mu}), (\ref{proof_suff:ineq:F_h_mu}) yield that $$\{(h,\mu,z)\in [s,r-1/n]\times\ptdu\times\mathbb{R}:z\in [z^*-(h-s),\psi_2(h,\mathrm{p}^1{}_\#\mu)]\}\subset\mathcal{E}(h,\mu). $$ Thus, $\{\mathcal{E}(h,\mu)\}_{h\in [s,r-1/n],\mu\in\ptd}$ is a cover of $\mathcal{K}$. Let us prove that it is an open cover. To this end we are to show that each set $\mathcal{E}(h,\mu)$ is open. Let $(t,\alpha,z)\in\mathcal{E}(h,\mu)$, $\varepsilon$ be a positive number. Pick $(t',\alpha',z')\in\mathcal{K}$  such that $|t-t'|^2+W_2^2(\alpha,\alpha')+|z-z'|^2<\varepsilon^2$. This implies that $|t-t'|,\ \ W_2(\alpha,\alpha'),\ \ |z-z'|<\varepsilon$. We shall show that, for sufficiently small $\varepsilon$, $(t',\alpha',z')\in\mathcal{E}(h,\mu)$. Pick $\eta$ such that conditions (E1)--(E4) are satisfied for $(t,\alpha,z)$ and $\eta$. Let $\pi$ be an optimal plan between $\alpha'$ and $\alpha$. Set $\eta'\triangleq \pi*\eta$. 
	
	Condition (E1) holds for $(t',\alpha',z')$ and $\eta'$ when $\varepsilon<\tau_{h,\mu}-|t-h|$.
	We have that if $$\varepsilon(1+1/n)<\psi_2(h^+,\mathrm{p}^1{}_\#\mu^+_{h,\mu})-z+(h^+_{h,\mu}-t)/n$$ then condition (E2) is fulfilled for $(t',\alpha',z')$ and $\eta'$. Further, by Lemma~\ref{lm:continuity_Xi}
	$$W_2(\Xi^{h^+_{h,\mu}-t'}{}_\#\eta',\Xi^{h^+_{h,\mu}-t}{}_\#\eta)\leq W_2(\alpha',\alpha)+|t'-t|c. $$ Thus, when $$\varepsilon(1+c)<(h^+_{t,\mu}-h)/n-W_2(\Xi^{h^+_{h,\mu}-t}{}_\#\eta,\mu^+_{h,\mu})$$ condition (E3) is valid for $(t',\alpha',z')$ and $\eta'$. Finally, by Lemma~\ref{lm:lipschtitz_F}
	\begin{equation*}
	\begin{split}
	\Bigl|\int_{\tdubc}\mathrm{dist}(w,F_1(t,x,\mathrm{p}^1{}_\#\alpha,u))\eta(d(x,u,w)&)\\- \int_{\tdubc}\mathrm{dist}(w,F_1(t',x,\mathrm{p}^1{}_\#\alpha',u&))\eta'(d(x,u,w))\Bigr|  \\&\leq \varpi(t'-t)+ 2LW_2(\alpha',\alpha).
	\end{split} 
	\end{equation*} Consequently, picking $\varepsilon$ so small that
	$$\varpi(\varepsilon)+2L\varepsilon<\frac{1}{n}-\int_{\tdubc}\mathrm{dist}(w,F_1(t,x,\mathrm{p}^1{}_\#\alpha),u)\eta(d(x,u,w)), $$ we get condition (E4) for $(t',\alpha',z')$ and $\eta'$.

	Since $\{\mathcal{E}(h,\mu)\}$ is an open cover of $\mathcal{K}$, we can find a finite number of pairs $\{(h_i,\mu_i)\}_{i=1}^{I}$ such that $$\mathcal{K}=\bigcup_{i=1}^{I} \mathcal{E}(h_i,\mu_i).$$ Put
	$$\varepsilon_n\triangleq \min_{i=1,\ldots I}\tau_{h_i,\mu_i}. $$ For $(t,\alpha)\in [s,r-1/n]\times\ptdu$ and $z=\psi_2(t,\mathrm{p}^1{}_\#\alpha)$ there exists a number $i\in \{1,\ldots,I\}$ such that $(t,\alpha,z)\in\mathcal{E}(h_i,\mu_i)$. Pick $\eta$ satisfying conditions (E1)--(E4) for $(t,\alpha,z)$ and $(h_i,\mu_i)$. Set $t^+\triangleq h^+_{h_i,\mu_i}$, $\alpha^+\triangleq \mu^+_{h_i,\mu_i}$.
\end{proof}
\begin{remark}\label{remark:psi_estimate} Notice that $(t^+,\alpha^+)$ constructed by Lemma~\ref{lm:one_step} is such that $\psi_2(t^+,\mathrm{p}^1{}_\#\alpha^+)\geq z^*-(t^+-s)$. 
\end{remark}

Let us introduce operations used to prove the sufficiency part of Theorem~\ref{th:characterisation}. First, given $t^1,t^2\in [0,T]$, define the operator $L^{t^1,t^2}:\tdu\times \rd\rightarrow\ctu{t^1}{t^2}$ by the rule:
$$L^{t^1,t^2}(y,u,w)\triangleq (x(\cdot),u), $$ with $x(t)\triangleq y+(t-t^1)w$.

Now, let $t^1<t^2<t^3$, $a_1=(x_1(\cdot),u_1)\in \ctu{t^1}{t^2}$, $a_2=(x_2(\cdot),u_2)\in \ctu{t^2}{t^3}$. Assume that  $x_1(t^2)=x_2(t^2)$, $u_1=u_2$. Concatenation of $a_1$, $a_2$ is a pair $a_1\odot a_2\triangleq(x(\cdot),u)\in\ctu{t^1}{t^3}$ such that $u=u_1=u_2$, whereas
$$x(t)\triangleq \left\{\begin{array}{cc}
x_1(t), & t\in [t^1,t^2] \\
x_2(t), & t\in [t^2,t^3].
\end{array}\right. $$

Further, let $\nu_1\in\pctu{t^1}{t^2}$, $\nu_2\in\pctu{t^2}{t^3}$ satisfy $\mu\triangleq\hae{t^2}{}_\#\nu_1=\hae{t^2}{}_\#\nu_2$. Let $\nu_2(\cdot|y,u)$ be the disintegration of $\nu_2$ along $\mu$ i.e. each $\nu_{2}(\cdot|y,u_2)$ is concentrated on the set of pairs $(x_2(\cdot),u)$ such that $x(t^2)=y$, $u=u_2$ and, for any $\phi\in C_b(\ctu{t^2}{t^3})$,
$$\int_{\ctu{t^2}{t^3}}\phi(a_2)\nu_2(da_2)= \int_{\tdu}\int_{\ctu{t^2}{t^3}}\phi(a_2)\nu_{2}(da_2|y,u)\mu(d(y,u)). $$ Define $\nu_1\odot\nu_2$ by the rule: for $\phi\in C_b(\ctu{t^1}{t^3})$,
$$\int_{\ctu{t^1}{t^3}}\phi(a)(\nu_1\odot\nu_2)(da)\triangleq \int_{\ctu{t^1}{t^2}}\int_{\ctu{t^2}{t^3}}\phi(a_1\odot a_2)\nu_{2}(da_2|\hae{t^2}(a_1))\nu_1(da_1). $$
Notice that, for $t\in [t^1,t^2]$,
$$\hae{t}{}_\#(\nu_1\odot\nu_2)\triangleq \hae{t}{}_\#\nu_1, $$ whereas when $t\in [t^2,t^3]$,
$$\hae{t}{}_\#(\nu_1\odot\nu_2)\triangleq \hae{t}{}_\#\nu_2. $$

\begin{proof}[Proof of Theorem~\ref{th:characterisation}. Sufficiency]

	Given a sufficiently large natural number $n$, we construct a number $J_n$ and sequences $\{t_n^j\}_{j=0}^{J_n}\subset [s,r]$, $\{\alpha_n^j\}_{j=0}^{J_n},\{\mu_n^j\}_{j=0}^{J_n}\subset \ptdu$, $\{\eta_n^j\}_{j=1}^{J_n}, \{\gamma_n^j\}_{j=1}^{J_n}\subset\ptdubc$ by the following rules.
	\begin{enumerate}
		\item Set $t_n^0\triangleq s$, $\mu_n^0=alpha_n^0\triangleq\\alpha_*$;
		\item If $t_n^j<r-1/n$, then  pick $t^+$, $\alpha^+$ and $\eta$ satisfying conditions of Lemma~\ref{lm:one_step} for $t=t_n^j$, $\alpha=\alpha_n^j$. Put $t_n^{j+1}\triangleq t^+$, $\alpha_n^{j+1}\triangleq \alpha^+$, $\eta_n^{j+1}\triangleq \eta$. Further, let $\pi_n^j\in \Pi^0(\mu_n^j,\alpha_n^j)$. Set $\gamma_n^{j+1}\triangleq\pi_n^j*\eta_n^{j+1}$, $\mu_n^{j+1}\triangleq \Xi^{t_n^{j+1}-t_n^j}{}_\#\gamma_n^{j+1}$.
		\item If $t_n^j\geq r-1/n$, then set $J_n\triangleq j$.
	\end{enumerate} Notice that, since $t_n^{j+1}-t_n^j\in [\varepsilon_n,1/n]$ where $\varepsilon_n$ is a positive number, this process is finite and $t_n^{J_n}\in [r-1/n,r)$. Furthermore, 
	\begin{equation*}
	\psi_2(t_n^{j+1},\mathrm{p}^1{}_\#\alpha_n^{j+1})\geq \psi_2(t_n^j,\mathrm{p}^1{}_\#\alpha_n^j) - (t_n^{j+1}-t_n^j)/n.
	\end{equation*} This gives that
	\begin{equation}\label{proof_suff:ineq:psi_alpha}
	\psi_2(t_n^{J_n},\mathrm{p}^1{}_\#\alpha_n^{J_n})\geq \psi_2(s,\mathrm{p}^1{}_\#\alpha_*)-(r-s)/n.
	\end{equation}
	
	The following inequality is fulfilled: \begin{equation}\label{proof_suff:ineq:alpha_mu_j}
	W_2(\alpha_n^j,\mu_n^j)\leq (t_n^j-s)/n.
	\end{equation} Indeed, (\ref{proof_suff:ineq:alpha_mu_j}) is obviously fulfilled for $j=0$. If it holds for some $j$, then we have that
	\begin{equation*}
	W_2(\alpha_n^{j+1},\mu_n^{j+1})\leq W_2(\Xi^{t_n^{j+1}-t_n^j}{}_\#\gamma_n^{j+1},\Xi^{t_n^{j+1}-t_n^j}{}_\#\eta_n^{j+1})+W_2(\Xi^{t_n^{j+1}-t_n^j}{}_\#\eta_n^{j+1},\mu_n^{j+1}).
	\end{equation*} Since  $\gamma_n^{j+1}\triangleq\pi_n^j*\eta_n^{j+1}$ and $\pi_n^j\in \Pi^0(\mu_n^j,\alpha_n^j)$, Lemma~\ref{lm:continuity_Xi} gives that $W_2(\Xi^{t_n^{j+1}-t_n^j}{}_\#\gamma_n^{j+1},\Xi^{t_n^{j+1}-t_n^j}{}_\#\eta_n^{j+1})\leq W_2(\alpha_n^j,\mu_n^j)$. Lemma~\ref{lm:one_step} states that $W_2(\Xi^{t_n^{j+1}-t_n^j}{}_\#\eta_n^{j+1},\mu_n^{j+1})\leq (t_n^{j+1}-t_n^j)/n$. Combining this estimates with the assumption, we obtain inequality (\ref{proof_suff:ineq:alpha_mu_j}) for $j+1$.
	
	Further, we claim that 
	\begin{equation}\label{proof_suf:ineq:F_estimate}
	\int_{\tdubc}\mathrm{dist}(w,F_1(t_n^j,x,\mathrm{p}^1{}_\#\gamma_n^{j+1},u))\gamma_n^{j+1}(d(x,u,w))\leq (2L(t_n^j-s)+1)/n.
	\end{equation} Indeed, by Lemma~\ref{lm:one_step} we have that
	\begin{equation}\label{proof_suf:ineq:F_1_estima_eta}
	\int_{\tdubc}\mathrm{dist}(w,F_1(t_n^j,x,\mathrm{p}^1{}_\#\eta_n^{j+1},u))\eta_n^{j+1}(d(x,u,w))\leq 1/n. 
	\end{equation} Since $\gamma^{j+1}_n=\pi_n^j*\eta^{j+1}_n$ and $\pi_n^j\in \Pi^0(\mu_n^j,\alpha_n^j)$, Lemma~\ref{lm:lipschtitz_F} yields that 
	\begin{equation*}
	\begin{split}
	\Bigl|\int_{\tdubc}\mathrm{dist}&(w,F_1(t_n^j,x,\mathrm{p}^1{}_\#\gamma_n^{j+1},u))\gamma_n^{j+1}(d(x,u,w))\\ &- \int_{\tdubc}\mathrm{dist}(w,F_1(t_n^j,x,\mathrm{p}^1{}_\#\eta_n^{j+1},u))\eta_n^{j+1}(d(x,u,w))\Bigr|\\&{}\hspace{200pt}\leq 2LW_2(\alpha_n^j,\mu_n^j).
	\end{split}
	\end{equation*}  This, (\ref{proof_suff:ineq:alpha_mu_j}) and (\ref{proof_suf:ineq:F_1_estima_eta}) imply (\ref{proof_suf:ineq:F_estimate}).
	
	Let $\gamma_n^{J_n+1}$ be a probability on $\tdubc$ such that $\mathrm{p}^{1,2}{}_\#\gamma_n^{J_n+1}=\mu_n^{J_n}$ and 
	\begin{equation}\label{proof_suf:ineq:choice_gamma_J_np}
	\int_{\tdubc}\mathrm{dist}(w,F_1(t_n^{J_n},x,\mathrm{p}^1{}_\#\gamma_n^{J_n+1},u))\gamma_n^{J_n+1}(d(x,u,w))\leq 1/n. 
	\end{equation} Set $t_n^{J_n+1}\triangleq r$, $\mu_n^{J_n+1}\triangleq \Xi^{t_n^{J_n+1}-t_n^{J_n}}{}_\#\gamma_n^{J_n+1}$. Since $r-t^{J_n}_n\leq 1/n$ and the probability $\gamma_n^{J_n+1}$ is concentrated on $\tdubc$, we have that $W_2(\mu_n^{J_n+1},\mu_n^{J_n})\leq c/n$. This and (\ref{proof_suff:ineq:alpha_mu_j}) yield the estimate
	\begin{equation}\label{proof_suff:ineq:mu_J_n_p_alpha_J_n}
	W_2(\mu_n^{J_n+1},\alpha_n^{J_n})\leq (r-s+c)/n.
	\end{equation}
	
	Now, define the probabilities on pairs consisting of motion and control. First, for $j=0,\ldots,J_n$, set $\nu_n^j\triangleq L^{t_n^j,t_{n}^{j+1}}{}_\#\gamma_n^{j+1}$. Notice that  $\hae{t_n^{j+1}}{}_\#\nu_n^j=\hae{t_n^{j+1}}{}_\#\nu_n^{j+1}$ if $j=0,\ldots, J_n-1$. Thus, the probability \begin{equation}\label{proof_suf:intro:nu_n_composition}
	\nu_n\triangleq \nu_n^0\odot\ldots\odot \nu_n^{J_n} 
	\end{equation} is well defined. Let us mention properties of $\nu_n$. First, we have that \begin{equation}\label{proof_suf:equality:t_n_nu_n_mu_n}
	\hae{t_n^j}{}_\#\nu_n=\mu_n^j,\ \ j=0,\ldots,J_n+1.
	\end{equation} Thus,
	\begin{equation}\label{proof_suf:equality:nu_n_alpha_star}
	\hae{s}{}_\#\nu_n=\alpha_*.
	\end{equation}
	Inequality (\ref{proof_suff:ineq:mu_J_n_p_alpha_J_n}) and equality (\ref{proof_suf:equality:t_n_nu_n_mu_n}) imply that
	\begin{equation}\label{proof_suf:ineq:nu_n_r_alpha_n}
	W_2(\hae{r}{}_\#\nu_n,\alpha_n^{J_n})\leq (r-s+c)/n.
	\end{equation} 
	
	Now, given $t',t''\in [s,r]$, let us evaluate the value
	$$\int_{\ctu{s}{r}}\mathrm{dist}\left(x(t'')-x(t'),\int_{t'}^{t''}F_1(t,x(t),\et{t}{}_\#\nu_n,u)dt \right)\nu_n(d(x(\cdot),u)). $$ Let $I_n'$ and $I_n''$ be such that $t'\in [t_n^{I_n'},t_n^{I_n'+1}]$, $t''\in [t_n^{I_n''},t_n^{I_n''+1}]$. Without loss of generality, we shall assume that $I_n'<I_n''$. Set $\tau_n^{I_n'}\triangleq t'$. For $j=I_n'+1,\ldots, I_n''$, let $\tau_n^j\triangleq t_n^j$. Finally, put $\tau_n^{I_n''+1}\triangleq t''$. We have that
	\begin{equation}\label{proof_suf:ineq:dist_F_1_first}
	\begin{split}
	\mathrm{dist}\Bigl(&x(t'')-x(t'),\int_{t'}^{t''}F_1(t,x(t),\et{t}{}_\#\nu_n,u)dt \Bigr) \\ &= \mathrm{dist}\left(\sum_{j=I_n'}^{I_n''}x(\tau_n^{j+1})-x(\tau_n^j),\sum_{j=I_n'}^{I_n''}\int_{\tau_n^j}^{\tau_n^{j+1}}F_1(t,x(t),\et{t}{}_\#\nu_n,u)dt\right) \\ &\leq
	\sum_{j=I_n'}^{I_n''}\mathrm{dist}\left(x(\tau_n^{j+1})-x(\tau_n^j), \int_{\tau_n^j}^{\tau_n^{j+1}}F_1(t,x(t),\et{t}{}_\#\nu_n,u)dt\right).
	\end{split}
	\end{equation}
	Recall that  $\nu_n$ is a concatenation of the probabilities $\nu_n^j$ (see (\ref{proof_suf:intro:nu_n_composition})), whereas each probability $\nu_n^j$ is concentrated on the set of pairs $(x(\cdot),u)\in\ctu{t_n^j}{t_n^{j+1}}$, where $x(\cdot)$ is $c$-Lipschitz continuous. Furthermore, by (\ref{proof_suf:equality:t_n_nu_n_mu_n}) we have that $W_2(\et{t}{}_\#\nu_n,\mathrm{p}^1{}_\#\mu_n^j)\leq c(t-t_n^j)$ when $t\in [\tau_n^j,\tau_n^{j+1}]$. This implies that
	\begin{equation}\label{proof_suf:ineq:dist_int_firts}
	\begin{split}
	\int_{\ctu{s}{r}}&\mathrm{dist}\left(x(\tau_n^{j+1})-x(\tau_n^j), \int_{\tau_n^j}^{\tau_n^{j+1}}F_1(t,x(t),\et{t}{}_\#\nu_n,u)dt\right)\nu_n(d(x(\cdot),u)) \\ &=\int_{\ctu{s}{r}}\mathrm{dist}\left(x(\tau_n^{j+1})-x(\tau_n^j), \int_{\tau_n^j}^{\tau_n^{j+1}}F_1(t,x(t),\et{t}{}_\#\nu_n,u)dt\right)\nu_n^j(d(x(\cdot),u)) \\ &\leq \int_{\ctu{s}{r}}\mathrm{dist}\left(x(\tau_n^{j+1})-x(\tau_n^j), \int_{\tau_n^j}^{\tau_n^{j+1}}F_1(t_n^j,x(t_n^j),\mu_n^j,u)dt\right)\nu_n^j(d(x(\cdot),u))\\ &{}\hspace{150pt}+(\tau_n^{j+1}-\tau_n^j)(\varpi(1/n)+2Lc/n).
	\end{split}
	\end{equation} Recall that $\nu_n^j\triangleq L^{t_n^j,t_n^{j+1}}{}_\#\gamma_n^{j+1}$. Therefore,
	\begin{equation*}\label{proof_suf:ineq:dist_int_second}
	\begin{split}
	\int_{\ctu{s}{r}}&\mathrm{dist}\left(x(\tau_n^{j+1})-x(\tau_n^j), \int_{\tau_n^j}^{\tau_n^{j+1}}F_1(t_n^j,x(t_n^j),\mu_n^j,u)dt\right)\nu_n^j(d(x(\cdot),u)) \\ 
	&= (\tau_n^{j+1}-\tau_n^j)\int_{\tdubc}\mathrm{dist}(w,F_1(t_n^j,x,\mu_n^j,u))\gamma_n^{j+1}(d(x,u,w)).
	\end{split}
	\end{equation*} This, (\ref{proof_suf:ineq:F_estimate}) and (\ref{proof_suf:ineq:choice_gamma_J_np}) give that
	\begin{equation*}
	\begin{split}
	\int_{\ctu{s}{r}}&\mathrm{dist}\left(x(\tau_n^{j+1})-x(\tau_n^j), \int_{\tau_n^j}^{\tau_n^{j+1}}F_1(t_n^j,x(t_n^j),\mu_n^j,u)dt\right)\nu_n^j(d(x(\cdot),u)) \\ 
	&\leq (\tau_n^{j+1}-\tau_n^j)(2L(r-s)+1)/n.
	\end{split}
	\end{equation*} Applying this estimate for the right-hand side of (\ref{proof_suf:ineq:dist_int_firts}), we get
	\begin{equation*}
	\begin{split}
	\int_{\ctu{s}{r}}&\mathrm{dist}\left(x(\tau_n^{j+1})-x(\tau_n^j), \int_{\tau_n^j}^{\tau_n^{j+1}}F_1(t,x(t),\et{t}{}_\#\nu_n,u)dt\right)\nu_n(d(x(\cdot),u)) \\ &\leq
	(\tau_n^{j+1}-\tau_n^j)(\varpi(1/n)+2Lc/n+2L(r-s)/n+1/n).
	\end{split}
	\end{equation*} Combining this inequality with (\ref{proof_suf:ineq:dist_F_1_first}), we conclude that
	\begin{equation}\label{proof_suf:ineq:int_F_1_final}
	\begin{split}
	\int_{\ctu{s}{r}}&\mathrm{dist}\left(x(t'')-x(t'),\int_{t'}^{t''}F_1(t,x(t),\et{t}{}_\#\nu_n,u) dt\right)\nu_n(d(x(\cdot),u))\\ &\leq (r-s)(\varpi(1/n)+2Lc/n+2L(r-s)/n+1/n).
	\end{split}
	\end{equation}

	By construction the probability $\nu_n$ is concentrated on the set of pairs $(x(\cdot),u)$ where $x(\cdot):[s,r]\rightarrow \td$ is $c$-Lipschitz continuous, $u\in U$. Thus, the sequence $\{\nu_n\}$ is relatively compact. Pick a subsequence $\{\nu_{n_k}\}$ and a probability $\nu\in\pctu{s}{r}$  such that
	\begin{equation}\label{proof_suf:convergence:nu_n_k}
	W_2(\nu_{n_k},\nu)\rightarrow 0\text{ as }k\rightarrow\infty. 
	\end{equation} We shall prove that $\nu$ satisfy conditions of Corollary~\ref{corollary:equivalence}. This will imply $v$-stability of $\psi_2$. 
	
	First, we have that $\hae{s}{}_\#\nu=\alpha_*$. Thus, the first condition of Corollary~\ref{corollary:equivalence} is valid. Further, passing to the limit in (\ref{proof_suf:ineq:int_F_1_final}) we obtain the second condition. To show the third condition, notice that by construction of $\nu$ (see (\ref{proof_suf:convergence:nu_n_k})) and (\ref{proof_suf:ineq:nu_n_r_alpha_n})
	$\{(t_{n_k}^{J_{n_k}},\alpha_{n_k}^{J_{n_k}})\}$ converges to $(r,\hae{r}{}_\#\nu)$. Taking into account (\ref{proof_suff:ineq:psi_alpha}) and upper semicontinuity of $\psi_2$, we get
	$$\psi_2(r,\et{r}{}_\#\nu)\geq \psi_2(s,\mathrm{p}^1{}_\#\alpha_*). $$ This, the third condition is also fulfilled. 
	
	Since $s$, $r$, and $\alpha$ are chosen arbitrarily, using Corollary~\ref{corollary:equivalence}, we conclude that $\psi_2$ is $v$-stable. 
\end{proof}

\section{Proof of Theorem~\ref{th:characterisation}. Necessity}\label{sect:proof:necessity}

Now we assume that $\psi_2:[0,T]\times \td \rightarrow \mathbb{R}$ is $v$-stable.  We shall prove that, for any $s\in [0,T]$, $m\in\ptd$, and $\alpha\in\ptdu$,$$\sup\{\vd_c\psi_2(t,\eta):\eta\in\mathcal{F}_1(s,\alpha)\}\geq 0, $$ where $c=R$.

For $s,r\in [0,T]$, $s<r$, let $\Delta^{s,r}$ be an operator from $\mathcal{C}_{s,r}\times U$ to  $\td\times U\times\rd$ acting by the rule:
$$\Delta^{s,r}(x(\cdot),u)\triangleq \left(x(s),u,\frac{x(r)-x(s)}{r-s}\right). $$

Notice that if $x(\cdot)$ is $c$-Lipschitz continuous, then $\Delta^{s,r}(x(\cdot),u)\in \td\times U\times \mathbb{B}_c$.

\begin{proof}[Proof of Theorem~\ref{th:characterisation}. Necessity]
	Since $\psi_2$ is $v$-stable, we have (see Corollary~\ref{corollary:equivalence}) that, for any $s,r\in [0,T]$, $s<r$, $\alpha\in\ptdu$, one can find $\nu\in\pctu{s}{r}$  such that
	\begin{enumerate}
		\item $e_t{}_\#\nu=\alpha$;
		\item for any $t',t''\in [s,r]$,
		$$\int_{\pctu{s}{r}}\mathrm{dist}\left(\et{t''}(a)-\et{t'}(a),\int_{t'}^{t''} F_1(t,\et{t}(a),\et{t}{}_\#\nu,\mathrm{p}^2(a))dt\right)\nu(d(a(\cdot)))=0; $$
		\item $\psi_2(s,\et{s}{}_\#\nu)\leq \psi_2(r,\et{r}{}_\#\nu).$
	\end{enumerate}
	
	Put $m\triangleq \mathrm{p}^1{}_\#\alpha$.
	
	Notice that $\nu$ is concentrated on the set of pairs $(x(\cdot),u)$, where $x(\cdot)$ is $R$-Lipschitz continuous.
	
	Put $\eta^{s,r}\triangleq \Delta^{s,r}{}_\#\nu$. We have that each probability $\eta^{s,r}$ is concentrated on the set $\td \times U\times\mathbb{B}_R$. Thus, the set $\{\eta^{s,r}\}_{r>s}$ is relatively compact in $\ptdubc[R]$.
	
	We have that there exist a sequence $\{r_n\}_{n=1}^\infty$ and a probability $\eta\in\ptdu$ such that  $r_n\rightarrow s$ and $W_2(\eta^{s,r_n},\eta)\rightarrow 0$ as $n\rightarrow\infty$. By construction $\mathrm{p}^{1,2}{}_\#\eta^{s,r}=\alpha$ for any $r>s$. Consequently, $\mathrm{p}^{1,2}{}_\#\eta=\alpha$. Since $\eta^{s,r_n}$ is concentrated on $\tdu\times \mathbb{B}_R$, $\Theta^{r_n-s}{}_\#\nu^{s,r_n}=\et{r_n}{}_\#\nu^{s,r}$ and $\psi_2(s,m)\geq \psi_2(r_n,\et{r_n}{}_\#\nu^{s,r})$, we have that
	\begin{equation}\label{proof_nec:ineq:derivative}
	\vd_R\psi_2(s,\eta)\geq 0.
	\end{equation} 
	
	Further, we have that 
	\begin{equation*}
	\begin{split}
	0&=\int_{\ctu{s}{r}}\mathrm{dist}\left(x(r)-x(s),\int_{s}^{r} F_1(t,x(t),\et{t}{}_\#\nu,u)dt\right)\nu(d(x(\cdot),u)) \\ &\geq
	\int_{\ctu{s}{r}}\mathrm{dist}\Bigl(x(r)-x(s),(r-s) F_1(s,x(s),m,u)\Bigr)\nu(d(x(\cdot),u))\\&{}\hspace{100pt}-(\alpha(r-s)+2LR(r-s))(r-s).
	\end{split}
	\end{equation*}
	
	Dividing both sides by $(r-s)$ and taking into account the definition of $\eta^{s,r}$, we get
	\begin{equation*}\label{proof_nec:ineq:dist_nu_s_r}
	\int_{\tdu\times\rd} \mathrm{dist}\left(w,F_1(s,x,m,u)dt\right)\eta^{s,r}(d(x,u,w))\leq \alpha(r-s)+2LR(r-s).
	\end{equation*} Letting $r=r_n$ and passing to the limit when $n\rightarrow\infty$, we conclude that
	$\eta\in\mathcal{F}_1(s,\alpha)$. This and (\ref{proof_nec:ineq:derivative}) imply the necessity part of the theorem.
\end{proof}

\bibliography{th_33_dirderivative}
\end{document}